\documentclass[12pt,a4paper]{article}
\usepackage{amssymb}
\usepackage{amsmath}
\usepackage{amsfonts}

\newtheorem{theorem}{Theorem}

\newtheorem{corollary}{Corollary}
\newtheorem{definition}{Definition}

\newtheorem{lemma}{Lemma}

\newenvironment{proof}[1][Proof]{\noindent\textbf{#1.} }{\ \rule{0.0em}{0.0em}}
\input{tcilatex}
\begin{document}

\title{\textbf{Bicomplex Version of Lebesgue's Dominated Convergence Theorem
and Hyperbolic Invariant Measure}}
\author{Chinmay Ghosh$^{1}$, Soumen Mondal$^{2}$ \\
%EndAName
$^{1}$Department of Mathematics\\
Kazi Nazrul University\\
Nazrul Road, P.O.- Kalla C.H.\\
Asansol-713340, West Bengal, India \\
chinmayarp@gmail.com \\
$^{2}$28, Dolua Dakshinpara Haridas Primary School\\
Beldanga, Murshidabad\\
Pin-742133\\
West Bengal, India\\
mondalsoumen79@gmail.com}
\date{}
\maketitle

\begin{abstract}
In this article we have studied bicomplex valued measurable functions on an
arbitrary measurable space. We have established the bicomplex version of
Lebesgue's dominated convergence theorem and some other results related to
this theorem. Also we have proved the bicomplex version of
Lebesgue-Radon-Nikodym theorem. Finally we have introduced the idea of
hyperbolic version of invariant measure.


\end{abstract}

\section{Introduction}

In $1882$ Corrado Segre $\cite{Seg}$ introduced a new number system called
bicomplex numbers. Unlike quaternions this number system is a commutative
generalization of complex numbers by four reals. The book of G. B. Price $%
\cite{Pr}$ is a good resource of the analysis of bicomplex numbers. Many
works have been done on bicomplex functional analysis. Few researchers have
worked on bicomplex dynamics, bicomplex topological modules. In this article
we have studied bicomplex valued measurable functions on an arbitrary
measurable space. We have established the bicomplex version of Lebesgue's
dominated convergence theorem and some other results related to this
theorem. Also we have proved the bicomplex version of Lebesgue-Radon-Nikodym
theorem. Finally we have introduced the idea of hyperbolic version of
invariant measure. To prove the results in the first two subsections in our
main results we have used the ideas of the book of W. Rudin $\cite{Ru}$ and
for the results in the last subsection we have used $\cite{Luz}$ and $\cite%
{Via}.$

\section{Basis definitions}

We denote the set of real and complex numbers by $\mathbb{R}$ and $\mathbb{C}
$ respectively. We may think three imaginary numbers $\mathbf{i}_{1},\mathbf{%
i}_{2}$ and $\mathbf{j}$ governed by the rules%
\begin{equation*}
\mathbf{i}_{1}^{2}=-1,\mathbf{i}_{2}^{2}=-1,\mathbf{j}^{2}=1
\end{equation*}%
\begin{eqnarray*}
\mathbf{i}_{1}\mathbf{i}_{2} &=&\mathbf{i}_{2}\mathbf{i}_{1}=\mathbf{j} \\
\mathbf{i}_{1}\mathbf{j} &=&\mathbf{ji}_{1}=-\mathbf{i}_{2} \\
\mathbf{i}_{2}\mathbf{j} &=&\mathbf{ji}_{2}=-\mathbf{i}_{1}.
\end{eqnarray*}%
Then we have two complex planes $\mathbb{C}\left( \mathbf{i}_{1}\right)
=\left\{ x+\mathbf{i}_{1}y:x,y\in \mathbb{R}\right\} $ and $\mathbb{C}\left( 
\mathbf{i}_{2}\right) =\left\{ x+\mathbf{i}_{2}y:x,y\in \mathbb{R}\right\} ,$
both of which are identical to $\mathbb{C}.$ Bicomplex numbers are defined
as $\zeta =z_{1}+\mathbf{i}_{2}z_{2}$ for $z_{1},z_{2}\in \mathbb{C}\left( 
\mathbf{i}_{1}\right) $. The set of all bicomplex numbers is denoted by $%
\mathbb{T}$. In particular if $z_{1}=x,z_{2}=\mathbf{i}_{1}y$ where $x,y\in 
\mathbb{R}$ we get $\zeta =x+\mathbf{j}y$ and these type of numbers are
called hyperbolic numbers or duplex numbers. The set of all hyperbolic
numbers is denoted by $\mathbb{D}$. For $\left( z_{1}+\mathbf{i}%
_{2}z_{2}\right) ,\left( w_{1}+\mathbf{i}_{2}w_{2}\right) \in \mathbb{T},$
the addition and multiplication are definde as%
\begin{eqnarray*}
\left( z_{1}+\mathbf{i}_{2}z_{2}\right) +\left( w_{1}+\mathbf{i}%
_{2}w_{2}\right) &=&\left( z_{1}+w_{1}\right) +\mathbf{i}_{2}\left(
z_{2}+w_{2}\right) \\
\left( z_{1}+\mathbf{i}_{2}z_{2}\right) \left( w_{1}+\mathbf{i}%
_{2}w_{2}\right) &=&\left( z_{1}w_{1}-z_{2}w_{2}\right) +\mathbf{i}%
_{2}\left( z_{1}w_{2}+z_{2}w_{1}\right) .
\end{eqnarray*}%
With these operations $\mathbb{T}$ forms a commutative ring with zero
divisors. The elements $z_{1}+\mathbf{i}_{2}z_{2}\in \mathbb{T}$ such that $%
z_{1}^{2}+z_{2}^{2}=0$ are the zero divisors. The interesting property of a
bicomplex number is its idempotent representation. Setting $\mathbf{e}_{1}=%
\frac{1+\mathbf{j}}{2}$ and $\mathbf{e}_{2}=\frac{1-\mathbf{j}}{2},$ we get%
\begin{equation*}
z_{1}+\mathbf{i}_{2}z_{2}=\left( z_{1}-\mathbf{i}_{1}z_{2}\right) \mathbf{e}%
_{1}+\left( z_{1}+\mathbf{i}_{1}z_{2}\right) \mathbf{e}_{2}.
\end{equation*}%
Many calculations become easier for this representation.

Throughout this article we will consider $\mathfrak{M}$ to be a $\sigma -$%
algebra in a set $X,$ unless stated otherwise.

\subsection{Partial order on $\mathbb{D}$}

The set of nonnegative hyperbolic numbers is%
\begin{equation*}
\mathbb{D}^{+}=\left\{ \nu _{1}\mathbf{e}_{1}+\nu _{2}\mathbf{e}_{2}:\nu
_{1},\nu _{2}\geq 0\right\} .
\end{equation*}%
A hyperbolic number $\zeta $ is said to be $\left( \text{strictly}\right) $
positive if $\zeta \in \mathbb{D}^{+}\backslash \left\{ 0\right\} .$The set
of nonnegative hyperbolic numbers is also defined as 
\begin{equation*}
\mathbb{D}^{+}=\left\{ x+y\mathbf{k}:x^{2}-y^{2}\geq 0,x\geq 0\right\} .
\end{equation*}

On the realization of $\mathbb{D}^{+},$\ M.E. Luna-Elizarraras et.al.$\cite%
{Lun}$ defined a partial order relation on $\mathbb{D}$. For two hyperbolic
numbers $\zeta _{1},\zeta _{2}$ the relation $\preceq _{_{\mathbb{D}}}$ is
defined as 
\begin{equation*}
\zeta _{1}\preceq _{_{\mathbb{D}}}\zeta _{2}\text{ if and only if }\zeta
_{2}-\zeta _{1}\in \mathbb{D}^{+}.
\end{equation*}%
One can check that this relation is reflexive, transitive and antisymmetric.
Therefore $\preceq _{_{\mathbb{D}}}$ is a partial order relation on $\mathbb{%
D}$. This partial order relation $\preceq _{_{\mathbb{D}}}$ on $\mathbb{D}$
is an extension of the total order relation $\leq $ on $\mathbb{R}.$ We say $%
\zeta _{1}\prec _{_{\mathbb{D}}}\zeta _{2}$ if $\zeta _{1}\preceq _{_{%
\mathbb{D}}}\zeta _{2}$ but $\zeta _{1}\neq \zeta _{2}.$ Also we say $\zeta
_{2}\succeq _{_{\mathbb{D}}}\zeta _{1}$ if $\zeta _{1}\preceq _{_{\mathbb{D}%
}}\zeta _{2}$ and $\zeta _{2}\succ _{_{\mathbb{D}}}\zeta _{1}$ if $\zeta
_{1}\prec _{_{\mathbb{D}}}\zeta _{2}.$

\begin{definition}
For any hyperbolic number $\zeta =\nu _{1}\mathbf{e}_{1}+\nu _{2}\mathbf{e}%
_{2},$ the $\mathbb{D}-$modulus of $\zeta $ is defined by%
\begin{equation*}
\left\vert \zeta \right\vert _{\mathbb{D}}=\left\vert \nu _{1}\mathbf{e}%
_{1}+\nu _{2}\mathbf{e}_{2}\right\vert _{\mathbb{D}}=\left\vert \nu
_{1}\right\vert \mathbf{e}_{1}+\left\vert \nu _{2}\right\vert \mathbf{e}%
_{2}\in \mathbb{D}^{+}
\end{equation*}%
where $\left\vert \nu _{1}\right\vert $ and $\left\vert \nu _{2}\right\vert $
are the usual modulus of real numbers.
\end{definition}

\begin{definition}
A subset $A$\ of $\mathbb{D}$ is said to be $\mathbb{D}-$bounded if there
exists $M\in \mathbb{D}^{+}$ such that $\left\vert \zeta \right\vert _{%
\mathbb{D}}\preceq _{_{\mathbb{D}}}M$ for any $\zeta \in A.$
\end{definition}

Set%
\begin{eqnarray*}
A_{1} &=&\left\{ x\in \mathbb{R}:\exists \text{ }y\in \mathbb{R},\text{ }x%
\mathbf{e}_{1}+y\mathbf{e}_{2}\in A\right\} , \\
A_{2} &=&\left\{ y\in \mathbb{R}:\exists \text{ }x\in \mathbb{R},\text{ }x%
\mathbf{e}_{1}+y\mathbf{e}_{2}\in A\right\} .
\end{eqnarray*}%
If $A$\ is $\mathbb{D}-$bounded then $A_{1}$ and $A_{2}$ are bounded subset
of $\mathbb{R}$.

\begin{definition}
For a $\mathbb{D}-$bounded subset $A$\ of $\mathbb{D},$ the \textbf{supremum}
of $A$\ with respect to the $\mathbb{D}-$modulus is defined by 
\begin{equation*}
\sup\nolimits_{\mathbb{D}}A=\sup A_{1}\mathbf{e}_{1}+\sup A_{2}\mathbf{e}%
_{2}.
\end{equation*}
\end{definition}

\begin{definition}
A sequence of hyperbolic numbers $\left\{ \zeta _{n}\right\} _{n\geq 1}$ is
said to be \textbf{convergent to }$\zeta \in \mathbb{D}$ if for $\varepsilon
\in \mathbb{D}^{+}\backslash \left\{ 0\right\} $ there exists $k\in \mathbb{N%
}$ such that%
\begin{equation*}
\left\vert \zeta _{n}-\zeta \right\vert _{\mathbb{D}}\prec _{_{\mathbb{D}%
}}\varepsilon .
\end{equation*}%
Then we write 
\begin{equation*}
\lim\limits_{n\rightarrow \infty }\zeta _{n}=\zeta .
\end{equation*}
\end{definition}

\begin{definition}
A sequence of hyperbolic numbers $\left\{ \zeta _{n}\right\} _{n\geq 1}$ is
said to be $\mathbb{D}-$\textbf{Cauchy sequence }$\zeta \in \mathbb{D}$ if
for $\varepsilon \in \mathbb{D}^{+}\backslash \left\{ 0\right\} $ $\exists $ 
$N\in \mathbb{N}$ such that%
\begin{equation*}
\left\vert \zeta _{N+m}-\zeta _{N}\right\vert _{\mathbb{D}}\prec _{_{\mathbb{%
D}}}\varepsilon
\end{equation*}%
for all $m=1,2,3,...$ .
\end{definition}

Note that a sequence of hyperbolic numbers $\left\{ \zeta _{n}\right\}
_{n\geq 1}$ is \textbf{convergent }if and only if it is a $\mathbb{D}-$%
\textbf{Cauchy sequence.}

\begin{definition}
A hyperbolic series $\dsum\limits_{n=1}^{\infty }\zeta _{n}$ is \textbf{%
convergent} if and only if its partial sums is a $\mathbb{D}-$Cauchy
sequence, i.e.\textbf{,} for any $\varepsilon \in \mathbb{D}^{+}\backslash
\left\{ 0\right\} $ $\exists $ $N\in \mathbb{N}$ such that 
\begin{equation*}
\left\vert \dsum\limits_{k=1}^{m}\zeta _{N+k}\right\vert _{\mathbb{D}}\prec
_{_{\mathbb{D}}}\varepsilon
\end{equation*}%
for any $m\in \mathbb{N}.$
\end{definition}

\begin{definition}
A hyperbolic series $\dsum\limits_{n=1}^{\infty }\zeta _{n}$ is $\mathbb{D}-$%
\textbf{absolutely convergent} if the series $\dsum\limits_{n=1}^{\infty
}\left\vert \zeta _{n}\right\vert _{\mathbb{D}}$ is convergent.
\end{definition}

Every $\mathbb{D}-$absolutely convergent series is convergent.

\section{Main Results}

In this section we have established our main results. We have arranged these
in three subsections. In the first subsection we have proved the bicomplex
version of Lebesgue's dominated convergence theorem. \ In the second
subsection we have focussed on the bicomplex version of
Lebesgue-Radon-Nikodym theorem and also we have established bicomplex
version of Hahn decomposition theorem. Finally in the last subsection we
have introduced the idea of hyperbolic version of invariant measure.

\subsection{Bicomplex Version of Lebesgue's Dominated Convergence Theorem}

\begin{definition}
\cite{1} Let $X$ be a measurable space then the bicomplex valued function $%
f=f_{1}\mathbf{e}_{1}+f_{2}\mathbf{e}_{2}$ is called $\mathbb{T}-$measurable
on $X$ if $f_{1}$ and $f_{2}$ are complex measurable functions on $X.$ In
particular if $f_{1}$ and $f_{2}$ are real measurable functions on $X$ then $%
f=f_{1}\mathbf{e}_{1}+f_{2}\mathbf{e}_{2}$ is called $\mathbb{D}-$measurable
function on $X.$
\end{definition}

For a bicomplex measurable function $f=f_{1}\mathbf{e}_{1}+f_{2}\mathbf{e}%
_{2}$ one can easily check that $\left\vert f\right\vert =\left\vert
f_{1}\right\vert \mathbf{e}_{1}+\left\vert f_{2}\right\vert \mathbf{e}_{2}$
is $\mathbb{D}-$measurable. Also for two $\mathbb{T}-$measurable functions $%
f $ and $g$ it is routine check up that $f+g$ and $fg$ are also $\mathbb{T}-$%
measurable functions.

\begin{theorem}
If $f$ is a $\mathbb{T}-$measurable function on a measurable space $X$ then
there is a $\mathbb{T}-$measurable function $\alpha =\alpha _{1}\mathbf{e}%
_{1}+\alpha _{2}\mathbf{e}_{2}$ such that $\left\vert \alpha _{1}\right\vert
=1,\left\vert \alpha _{2}\right\vert =1$ and $f=\alpha \left\vert
f\right\vert .$
\end{theorem}

\begin{proof}
Let $f=f_{1}\mathbf{e}_{1}+f_{2}\mathbf{e}_{2}.$

Since $f$ is a $\mathbb{T}-$measurable function on a measurable space $X,$ $%
f_{1}$ and $f_{2}$ are complex measurable functions on $X.$

So there exist complex measurable functions $\alpha _{1},\alpha _{2}$ such
that $\left\vert \alpha _{1}\right\vert =1,\left\vert \alpha _{2}\right\vert
=1$ and $f_{1}=\alpha _{1}\left\vert f_{1}\right\vert ,f_{2}=\alpha
_{2}\left\vert f_{2}\right\vert .$

Set $\alpha =\alpha _{1}\mathbf{e}_{1}+\alpha _{2}\mathbf{e}_{2}.$ Obviously 
$\alpha $ is a $\mathbb{T}-$measurable function on $X$ and the result
follows.
\end{proof}

\begin{definition}
Let $\mathfrak{M}$ be a $\sigma -$algebra in a set $X.$ A bicomplex function 
$\mu =\mu _{1}\mathbf{e}_{1}+\mu _{2}\mathbf{e}_{2}$ defined on $X$ is
called \ a $\mathbb{T}-$measure on $\mathfrak{M}$ if $\mu _{1}\mathbf{,}\mu
_{2}$ are complex measures on $\mathfrak{M.}$ In particular if $\mu _{1}%
\mathbf{,}\mu _{2}$ are positive measures on $\mathfrak{M}$ i.e range of
both $\mu _{1}\mathbf{,}\mu _{2}$ are $\left[ 0,\infty \right] $ then $\mu $
is called a $\mathbb{D}-$measure on $\mathfrak{M}$ and if $\mu _{1}\mathbf{,}%
\mu _{2}$ are real measures on $\mathfrak{M}$ i.e range of both $\mu _{1}%
\mathbf{,}\mu _{2}$ are $[0,\infty )$ then $\mu $ is called a $\mathbb{D}%
^{+}-$measure on $\mathfrak{M.}$
\end{definition}

\begin{definition}
Let $\mu =\mu _{1}\mathbf{e}_{1}+\mu _{2}\mathbf{e}_{2}$ be a $\mathbb{D}-$%
measure on an arbitrary measurable space $X.$ We say $f=f_{1}\mathbf{e}%
_{1}+f_{2}\mathbf{e}_{2}$ to be bicomplex Lebesgue integrable function on $X$
if 
\begin{eqnarray*}
\dint\limits_{X}\left\vert f_{1}\right\vert d\mu _{1} &<&\infty , \\
\dint\limits_{X}\left\vert f_{2}\right\vert d\mu _{2} &<&\infty
\end{eqnarray*}%
i.e., $f_{i}\in L^{1}\left( \mu _{i}\right) ,$ the set of all complex
Lebesgue integrable functions with respect to $\mu _{i}$ for $i=1,2.$
\end{definition}

In that case we write%
\begin{equation*}
\dint\limits_{X}fd\mu =\left( \dint\limits_{X}f_{1}d\mu _{1}\right) \mathbf{e%
}_{1}+\left( \dint\limits_{X}f_{2}d\mu _{2}\right) \mathbf{e}_{2}.
\end{equation*}%
We define $L_{\mathbb{T}}^{1}\left( \mu \right) $ to be the collection of
all bicomplex Lebesgue integrable functions on $X.$

\begin{theorem}
Let $f,g\in L_{\mathbb{T}}^{1}\left( \mu \right) $ and $\alpha ,\beta \in 
\mathbb{C}.$ Then $\alpha f+\beta g\in L_{\mathbb{T}}^{1}\left( \mu \right)
, $ and%
\begin{equation*}
\dint\limits_{X}\left( \alpha f+\beta g\right) d\mu =\alpha
\dint\limits_{X}fd\mu +\beta \dint\limits_{X}gd\mu .
\end{equation*}
\end{theorem}

\begin{proof}
Let $f=f_{1}\mathbf{e}_{1}+f_{2}\mathbf{e}_{2},g=g_{1}\mathbf{e}_{1}+g_{2}%
\mathbf{e}_{2},\mu =\mu _{1}\mathbf{e}_{1}+\mu _{2}\mathbf{e}_{2},\alpha
=\alpha _{1}\mathbf{e}_{1}+\alpha _{2}\mathbf{e}_{2},\beta =\beta _{1}%
\mathbf{e}_{1}+\beta _{2}\mathbf{e}_{2}.$

Then%
\begin{equation*}
\alpha f+\beta g=\left( \alpha _{1}f_{1}+\beta _{1}g_{1}\right) \mathbf{e}%
_{1}+\left( \alpha _{2}f_{2}+\beta _{2}g_{2}\right) \mathbf{e}_{2}.
\end{equation*}%
Since $f,g\in L_{\mathbb{T}}^{1}\left( \mu \right) $ we have $\alpha
_{1}f_{1}+\beta _{1}g_{1}\in L^{1}\left( \mu _{1}\right) $ and $\alpha
_{2}f_{2}+\beta _{2}g_{2}\in L^{1}\left( \mu _{2}\right) $ and therefore $%
\alpha f+\beta g\in L_{\mathbb{T}}^{1}\left( \mu \right) .$

The last part follows from the facts%
\begin{eqnarray*}
\dint\limits_{X}\left( \alpha _{1}f_{1}+\beta _{1}g_{1}\right) d\mu _{1}
&=&\alpha _{1}\dint\limits_{X}f_{1}d\mu _{1}+\beta
_{1}\dint\limits_{X}g_{1}d\mu _{1}, \\
\dint\limits_{X}\left( \alpha _{2}f_{2}+\beta _{2}g_{2}\right) d\mu _{2}
&=&\alpha _{2}\dint\limits_{X}f_{2}d\mu _{2}+\beta
_{2}\dint\limits_{X}g_{2}d\mu _{2},
\end{eqnarray*}%
and%
\begin{equation*}
\dint\limits_{X}\left( \alpha f+\beta g\right) d\mu =\left(
\dint\limits_{X}\left( \alpha _{1}f_{1}+\beta _{1}g_{1}\right) d\mu
_{1}\right) \mathbf{e}_{1}+\left( \dint\limits_{X}\left( \alpha
_{2}f_{2}+\beta _{2}g_{2}\right) d\mu _{2}\right) \mathbf{e}_{2}.
\end{equation*}
\end{proof}

\begin{theorem}
If $f\in L_{\mathbb{T}}^{1}\left( \mu \right) ,$ then%
\begin{equation*}
\left\vert \dint\limits_{X}fd\mu \right\vert _{\mathbb{D}}\preceq _{\mathbb{D%
}}\dint\limits_{X}\left\vert f\right\vert _{\mathbb{D}}d\mu .
\end{equation*}
\end{theorem}

\begin{proof}
Let $f=f_{1}\mathbf{e}_{1}+f_{2}\mathbf{e}_{2},\mu =\mu _{1}\mathbf{e}%
_{1}+\mu _{2}\mathbf{e}_{2}.$

Since $f\in L_{\mathbb{T}}^{1}\left( \mu \right) ,$ we have $f_{1}\in
L^{1}\left( \mu _{1}\right) $ and $f_{2}\in L^{1}\left( \mu _{2}\right) .$

Therefore,%
\begin{eqnarray*}
\left\vert \dint\limits_{X}fd\mu \right\vert _{\mathbb{D}} &=&\left\vert
\left( \dint\limits_{X}f_{1}d\mu _{1}\right) \mathbf{e}_{1}+\left(
\dint\limits_{X}f_{2}d\mu _{2}\right) \mathbf{e}_{2}\right\vert _{\mathbb{D}}
\\
&=&\left\vert \dint\limits_{X}f_{1}d\mu _{1}\right\vert \mathbf{e}%
_{1}+\left\vert \dint\limits_{X}f_{2}d\mu _{2}\right\vert \mathbf{e}%
_{2}\preceq _{\mathbb{D}}\dint\limits_{X}\left\vert f_{1}\right\vert d\mu
_{1}\mathbf{e}_{1}+\dint\limits_{X}\left\vert f_{2}\right\vert d\mu _{2}%
\mathbf{e}_{2}=\dint\limits_{X}\left\vert f\right\vert _{\mathbb{D}}d\mu .
\end{eqnarray*}
\end{proof}

\begin{theorem}[Lebesgue's Dominated Convergence Theorem]
Let $\left\{ f_{n}=f_{n1}\mathbf{e}_{1}+f_{n2}\mathbf{e}_{2}\right\} $ be a
sequence of $\mathbb{T}-$measurable functions on $X$ such that%
\begin{equation*}
\lim\limits_{n\rightarrow \infty }f_{n}\left( x\right) =f\left( x\right)
\end{equation*}%
exists for all $x\in X.$ If there exists $g=g_{1}\mathbf{e}_{1}+g_{2}\mathbf{%
e}_{2}\in L_{\mathbb{T}}^{1}\left( \mu \right) $ such that%
\begin{equation*}
\left\vert f_{ni}\left( x\right) \right\vert \leq g_{i}\left( x\right)
\end{equation*}%
for all $n=1,2,3,...$ $;$ $i=1,2;$ $x\in X,$ then $f\in L_{\mathbb{T}%
}^{1}\left( \mu \right) ,$%
\begin{equation*}
\lim\limits_{n\rightarrow \infty }\dint\limits_{X}\left\vert f_{n}\left(
x\right) -f\left( x\right) \right\vert d\mu =0,
\end{equation*}%
and%
\begin{equation*}
\lim\limits_{n\rightarrow \infty }\dint\limits_{X}f_{n}\left( x\right) d\mu
=f\left( x\right) .
\end{equation*}
\end{theorem}

\begin{proof}
Since $\left\{ f_{n}\right\} $ is a sequence of $\mathbb{T}-$measurable
functions on $X,$ both $\left\{ f_{n1}\right\} $\textbf{\ }and $\left\{
f_{n2}\right\} $ are sequences of complex measurable functions on $X.$

Let $f=f_{1}\mathbf{e}_{1}+f_{2}\mathbf{e}_{2},\mu =\mu _{1}\mathbf{e}%
_{1}+\mu _{2}\mathbf{e}_{2}.$

Thus 
\begin{equation*}
\lim\limits_{n\rightarrow \infty }f_{ni}\left( x\right) =f_{i}\left( x\right)
\end{equation*}%
for $i=1,2.$

Now since 
\begin{equation*}
\left\vert f_{ni}\left( x\right) \right\vert \leq g_{i}\left( x\right)
\end{equation*}%
for all $n=1,2,3,...$ $;$ $i=1,2;$ $x\in X,$ we get $f_{i}\in L^{1}\left(
\mu _{i}\right) $ for $i=1,2$ and therefore $f\in L_{\mathbb{T}}^{1}\left(
\mu \right) .$

Also,%
\begin{equation*}
\lim\limits_{n\rightarrow \infty }\dint\limits_{X}\left\vert f_{ni}\left(
x\right) -f_{i}\left( x\right) \right\vert d\mu _{i}=0,
\end{equation*}%
and%
\begin{equation*}
\lim\limits_{n\rightarrow \infty }\dint\limits_{X}f_{n_{i}}\left( x\right)
d\mu _{i}=f_{i}\left( x\right)
\end{equation*}%
for $i=1,2.$ Hence%
\begin{eqnarray*}
&&\lim\limits_{n\rightarrow \infty }\dint\limits_{X}\left\vert f_{n}\left(
x\right) -f\left( x\right) \right\vert _{\mathbb{D}}d\mu \\
&=&\lim\limits_{n\rightarrow \infty }\left( \dint\limits_{X}\left\vert
f_{n1}\left( x\right) -f_{1}\left( x\right) \right\vert d\mu _{1}\right) 
\mathbf{e}_{1}+\lim\limits_{n\rightarrow \infty }\left(
\dint\limits_{X}\left\vert f_{n2}\left( x\right) -f_{2}\left( x\right)
\right\vert d\mu _{2}\right) \mathbf{e}_{2} \\
&=&0
\end{eqnarray*}%
and%
\begin{eqnarray*}
&&\lim\limits_{n\rightarrow \infty }\dint\limits_{X}f_{n}\left( x\right) d\mu
\\
&=&\lim\limits_{n\rightarrow \infty }\left( \dint\limits_{X}f_{n1}\left(
x\right) d\mu _{1}\right) \mathbf{e}_{1}+\lim\limits_{n\rightarrow \infty
}\left( \dint\limits_{X}f_{n2}\left( x\right) d\mu _{2}\right) \mathbf{e}_{2}
\\
&=&f_{1}\left( x\right) \mathbf{e}_{1}+f_{2}\left( x\right) \mathbf{e}_{2} \\
&=&f\left( x\right) .
\end{eqnarray*}
\end{proof}

\subsection{Bicomplex Version of Lebesgue-Radon-Nikodym Theorem}

Let $\mathfrak{M}$ be a measure space and $E\in \mathfrak{M}.$ Let $%
P=\left\{ E_{k}\right\} $ be a partition of $E.$ Then for all $E\in 
\mathfrak{M}$ the $\mathbb{T}-$measure $\mu =\mu _{1}\mathbf{e}_{1}+\mu _{2}%
\mathbf{e}_{2}$ on $\mathfrak{M}$ satisfies%
\begin{equation*}
\mu \left( E\right) =\dsum\limits_{k=1}^{\infty }\mu \left( E_{k}\right)
\end{equation*}%
for every partition $\left\{ E_{k}\right\} $ of $E.$

Let $\lambda =\lambda _{1}\mathbf{e}_{1}+\lambda _{2}\mathbf{e}_{2}$ be a $%
\mathbb{D}-$measure on $\mathfrak{M}.$ We say that $\lambda $ dominates $\mu 
$ on $\mathfrak{M}$\ if $\left\vert \mu _{i}\left( E\right) \right\vert \leq
\lambda _{i}\left( E\right) $ for all $E\in \mathfrak{M}$ and for $i=1,2.$
The $\mathbb{D}-$modulus of $\mu ,$ denoted by $\left\vert \mu \right\vert _{%
\mathbb{D}},$ is defined on $\mathfrak{M}$ by%
\begin{equation*}
\left\vert \mu \right\vert _{\mathbb{D}}\left( E\right)
=\sup_{P}\dsum\limits_{k=1}^{\infty }\left\vert \mu \left( E_{k}\right)
\right\vert _{\mathbb{D}}
\end{equation*}%
for all $E\in \mathfrak{M}.$

\begin{theorem}
For a $\mathbb{T}-$measure $\mu $ on $\mathfrak{M},$ $\left\vert \mu
\right\vert _{\mathbb{D}}$ is a $\mathbb{D}-$measure on $\mathfrak{M}.$
\end{theorem}

\begin{proof}
Let $\mu =\mu _{1}\mathbf{e}_{1}+\mu _{2}\mathbf{e}_{2}.$

Since $\mu _{1}\mathbf{,}\mu _{2}$ being complex measures on $\mathfrak{M},$ 
$\left\vert \mu _{1}\right\vert $\textbf{\ }and $\left\vert \mu
_{2}\right\vert $ are positive measures on $\mathfrak{M}.$

Hence $\left\vert \mu \right\vert _{\mathbb{D}}=\left\vert \mu
_{1}\right\vert \mathbf{e}_{1}+\left\vert \mu _{2}\right\vert \mathbf{e}%
_{2}\,$is a $\mathbb{D}-$measure on $\mathfrak{M}.$
\end{proof}

\begin{theorem}
For a $\mathbb{T}-$measure $\mu $ on $X,$%
\begin{equation*}
\left\vert \mu \right\vert _{\mathbb{D}}\left( X\right) \prec _{\mathbb{D}%
}\infty _{\mathbb{D}}.
\end{equation*}
\end{theorem}

\begin{proof}
Let $\mu =\mu _{1}\mathbf{e}_{1}+\mu _{2}\mathbf{e}_{2}.$

Since $\mu _{1}\mathbf{,}\mu _{2}$ being complex measures on $\mathfrak{M},$%
\begin{equation*}
\left\vert \mu _{i}\right\vert \left( X\right) <\infty
\end{equation*}%
for $i=1,2.$

Hence 
\begin{equation*}
\left\vert \mu \right\vert _{\mathbb{D}}\left( X\right) =\left\vert \mu
_{1}\right\vert \left( X\right) \mathbf{e}_{1}+\left\vert \mu
_{2}\right\vert \left( X\right) \mathbf{e}_{2}\prec _{\mathbb{D}}\infty _{%
\mathbb{D}}.
\end{equation*}
\end{proof}

Let $\mu ,\lambda $ be two $\mathbb{T}-$measures on $\mathfrak{M}$ and $c\in 
\mathbb{D}.$ For all $E\in \mathfrak{M}$ define%
\begin{eqnarray*}
\left( \mu +\lambda \right) \left( E\right) &=&\mu \left( E\right) +\lambda
\left( E\right) , \\
\left( c\mu \right) \left( E\right) &=&c\mu \left( E\right) .
\end{eqnarray*}%
One can easily check that $\mu +\lambda $ and $c\mu $ are also $\mathbb{T}-$%
measures on $\mathfrak{M.}$ The collection of all $\mathbb{T}-$measures on $%
\mathfrak{M}$ forms a module space over $\mathbb{D}.$

Slight modifying the definition of hyperbolic valued signed measure from $%
\cite{Gh},$ we now define it to be a $\mathbb{T}-$measure on $\mathfrak{M}$
having range in $\mathbb{D}^{+}\cup \mathbb{D}^{-}.$ Let $\mu $ be a signed $%
\mathbb{D}-$measure on $\mathfrak{M}$. Then both $\mu ^{+}=\frac{1}{2}\left(
\left\vert \mu \right\vert _{\mathbb{D}}+\mu \right) $ and $\mu ^{-}=\frac{1%
}{2}\left( \left\vert \mu \right\vert _{\mathbb{D}}-\mu \right) $ are $%
\mathbb{D}-$measures on $\mathfrak{M}.$ Obviously $\mu ^{+}$ and $\mu ^{-}$
are $\mathbb{D}-$bounded. Also the Jordan decomposition of a signed $\mathbb{%
D}-$measure is given by%
\begin{eqnarray*}
\mu &=&\mu ^{+}-\mu ^{-}, \\
\left\vert \mu \right\vert _{\mathbb{D}} &=&\mu ^{+}+\mu ^{-}.
\end{eqnarray*}

\begin{definition}
Let $\mu =\mu _{1}\mathbf{e}_{1}+\mu _{2}\mathbf{e}_{2}$ be a $\mathbb{D}-$%
measure and $\lambda =\lambda _{1}\mathbf{e}_{1}+\lambda _{2}\mathbf{e}_{2}$
be a $\mathbb{T}-$measure on $\mathfrak{M.}$ Then $\lambda $ is said to be
absolutely $\mathbb{T}-$continuous with respect to $\mu $ if $\lambda _{i}$
is absolutely continuous with respect to $\mu _{i}$ for $i=1,2.$ We denote
this by $\lambda \ll _{\mathbb{T}}\mu .$

If for $A\in \mathfrak{M},$ $\lambda _{i}$ is concentrated on $A$ for $%
i=1,2, $ then $\lambda $ is said to be $\mathbb{T}-$concentrated on $A.$

Two $\mathbb{T}-$measures $\lambda ^{\prime }=\lambda _{1}^{\prime }\mathbf{e%
}_{1}+\lambda _{2}^{\prime }\mathbf{e}_{2},\lambda ^{\prime \prime }=\lambda
_{1}^{\prime \prime }\mathbf{e}_{1}+\lambda _{2}^{\prime \prime }\mathbf{e}%
_{2}$\ on $\mathfrak{M}$ are called mutually $\mathbb{T}-$singular if $%
\lambda _{i}^{\prime }$ and $\lambda _{i}^{\prime \prime }$\ are mutually
singular for $i=1,2$. We denote this by $\lambda ^{\prime }\perp _{\mathbb{T}%
}\lambda ^{\prime \prime }.$
\end{definition}

\begin{theorem}
Let $\lambda ,\lambda ^{\prime }$ and $\lambda ^{\prime \prime }$ be $%
\mathbb{T}-$measures on $\mathfrak{M}.$ Also let $\mu $ be a $\mathbb{D}-$%
measure on $\mathfrak{M}.$ Then the following hold:

a) If $\lambda $ is $\mathbb{T}-$concentrated on $A,$ then $\left\vert
\lambda \right\vert _{\mathbb{D}}$ is also so.

b) If $\lambda ^{\prime }\perp _{\mathbb{T}}\lambda ^{\prime \prime },$ then 
$\left\vert \lambda ^{\prime }\right\vert _{\mathbb{D}}\perp _{\mathbb{T}%
}\left\vert \lambda ^{\prime \prime }\right\vert _{\mathbb{D}}.$

c) If $\lambda ^{\prime }\perp _{\mathbb{T}}\mu ,$ $\lambda ^{\prime \prime
}\perp _{\mathbb{T}}\mu ,$ then $\lambda ^{\prime }+$ $\lambda ^{\prime
\prime }\perp _{\mathbb{T}}\mu .$

d) If $\lambda ^{\prime }\ll _{\mathbb{T}}\mu ,$ $\lambda ^{\prime \prime
}\ll _{\mathbb{T}}\mu ,$ then $\lambda ^{\prime }+$ $\lambda ^{\prime \prime
}\ll _{\mathbb{T}}\mu .$

e) If $\lambda \ll _{\mathbb{T}}\mu ,$ then $\left\vert \lambda \right\vert
_{\mathbb{D}}\ll _{\mathbb{T}}\mu .$

f) If $\lambda ^{\prime }\ll _{\mathbb{T}}\mu ,$ $\lambda ^{\prime \prime
}\perp _{\mathbb{T}}\mu ,$ then $\lambda ^{\prime }\perp _{\mathbb{T}}$ $%
\lambda ^{\prime \prime }.$

g) If $\lambda \ll _{\mathbb{T}}\mu $ and $\lambda \perp _{\mathbb{T}}\mu $
then $\lambda =0.$
\end{theorem}

\begin{proof}
Let $\lambda =\lambda _{1}\mathbf{e}_{1}+\lambda _{2}\mathbf{e}_{2},$ $%
\lambda ^{\prime }=\lambda _{1}^{\prime }\mathbf{e}_{1}+\lambda _{2}^{\prime
}\mathbf{e}_{2},$ $\lambda ^{\prime \prime }=\lambda _{1}^{\prime \prime }%
\mathbf{e}_{1}+\lambda _{2}^{\prime \prime }\mathbf{e}_{2}$ and $\mu =\mu
_{1}\mathbf{e}_{1}+\mu _{2}\mathbf{e}_{2}.$

a) $\lambda $ is $\mathbb{T}-$concentrated on $A$ implies $\lambda _{i}$ is
concentrated on $A$ for $i=1,2.$

Thus $\left\vert \lambda _{i}\right\vert $ is concentrated on $A$ for $%
i=1,2. $

Therefore $\left\vert \lambda \right\vert _{\mathbb{D}}=\left\vert \lambda
_{1}\right\vert \mathbf{e}_{1}+\left\vert \lambda _{2}\right\vert \mathbf{e}%
_{2}$ is $\mathbb{T}-$concentrated on $A.$

b) $\lambda ^{\prime }\perp _{\mathbb{T}}\lambda ^{\prime \prime }$ implies $%
\lambda _{i}^{\prime }$ and $\lambda _{i}^{\prime \prime }$\ are mutually
singular for $i=1,2$.

Thus $\left\vert \lambda _{i}^{\prime }\right\vert $ and $\left\vert \lambda
_{i}^{\prime \prime }\right\vert $\ are mutually singular for $i=1,2$.

Therefore $\left\vert \lambda ^{\prime }\right\vert _{\mathbb{D}}=\left\vert
\lambda _{1}^{\prime }\right\vert \mathbf{e}_{1}+\left\vert \lambda
_{2}^{\prime }\right\vert \mathbf{e}_{2}$ and $\left\vert \lambda ^{\prime
\prime }\right\vert _{\mathbb{D}}=\left\vert \lambda _{1}^{\prime \prime
}\right\vert \mathbf{e}_{1}+\left\vert \lambda _{2}^{\prime \prime
}\right\vert \mathbf{e}_{2}$ are mutually $\mathbb{T}-$singular.

c) $\lambda ^{\prime }\perp _{\mathbb{T}}\mu ,$ $\lambda ^{\prime \prime
}\perp _{\mathbb{T}}\mu $ implies $\lambda _{i}^{\prime }$ and $\mu _{i}$\
are mutually singular for $i=1,2$ and $\lambda _{i}^{\prime \prime }$ and $%
\mu _{i}$\ are mutually singular for $i=1,2.$

Thus $\lambda _{i}^{\prime }+\lambda _{i}^{\prime \prime }$ and $\mu _{i}$\
are mutually singular for $i=1,2.$

Therefore $\lambda ^{\prime }+$ $\lambda ^{\prime \prime }=\left( \lambda
_{1}^{\prime }+\lambda _{1}^{\prime \prime }\right) \mathbf{e}_{1}+\left(
\lambda _{2}^{\prime }+\lambda _{2}^{\prime \prime }\right) \mathbf{e}_{2}$
and $\mu $\ are mutually $\mathbb{T}-$singular.

d) $\lambda ^{\prime }\ll _{\mathbb{T}}\mu ,$ $\lambda ^{\prime \prime }\ll
_{\mathbb{T}}\mu $ implies $\lambda _{i}^{\prime }$ is absolutely continuous
with respect to $\mu _{i}$ for $i=1,2$ and $\lambda _{i}^{\prime \prime }$
is absolutely continuous with respect to $\mu _{i}$ for $i=1,2.$

Thus $\lambda _{i}^{\prime }+\lambda _{i}^{\prime \prime }$ is absolutely
continuous with respect to $\mu _{i}$ for $i=1,2.$

Therefore $\lambda ^{\prime }+$ $\lambda ^{\prime \prime }=\left( \lambda
_{1}^{\prime }+\lambda _{1}^{\prime \prime }\right) \mathbf{e}_{1}+\left(
\lambda _{2}^{\prime }+\lambda _{2}^{\prime \prime }\right) \mathbf{e}_{2}$
is absolutely $\mathbb{T}-$continuous with respect to $\mu .$

e) $\lambda \ll _{\mathbb{T}}\mu $ implies $\lambda _{i}$ is absolutely
continuous with respect to $\mu _{i}$ for $i=1,2.$

Thus $\left\vert \lambda _{i}\right\vert $ is absolutely continuous with
respect to $\mu _{i}$ for $i=1,2.$

Therefore $\left\vert \lambda \right\vert _{\mathbb{D}}=\left\vert \lambda
_{1}\right\vert \mathbf{e}_{1}+\left\vert \lambda _{2}\right\vert \mathbf{e}%
_{2}$ is absolutely $\mathbb{T}-$continuous with respect to $\mu .$

f) $\lambda ^{\prime }\ll _{\mathbb{T}}\mu $ implies $\lambda _{i}^{\prime }$
is absolutely continuous with respect to $\mu _{i}$ for $i=1,2$ and $\lambda
^{\prime \prime }\perp _{\mathbb{T}}\mu $ implies $\lambda _{i}^{\prime
\prime }$ and $\mu _{i}$\ are mutually singular for $i=1,2.$

Thus $\lambda _{i}^{\prime }$ and $\lambda _{i}^{\prime \prime }$\ are
mutually singular for $i=1,2.$

Therefore $\lambda ^{\prime }=\lambda _{1}^{\prime }\mathbf{e}_{1}+\lambda
_{2}^{\prime }\mathbf{e}_{2}$ and $\lambda ^{\prime \prime }=\lambda
_{1}^{\prime \prime }\mathbf{e}_{1}+\lambda _{2}^{\prime \prime }\mathbf{e}%
_{2}$ are mutually $\mathbb{T}-$singular.

g) $\lambda \ll _{\mathbb{T}}\mu $ implies $\lambda _{i}$ is absolutely
continuous with respect to $\mu _{i}$ for $i=1,2$ and $\lambda \perp _{%
\mathbb{T}}\mu $ implies $\lambda _{i}$ and $\mu _{i}$\ are mutually
singular for $i=1,2.$

Thus $\lambda _{i}=0$ for $i=1,2.$

Therefore $\lambda =\lambda _{1}\mathbf{e}_{1}+\lambda _{2}\mathbf{e}_{2}=0.$
\end{proof}

\begin{theorem}[Lebesgue-Radon-Nikodym Theorem]
Let $\mu $ be a $\sigma -$finite $\mathbb{D}-$measure on $\mathfrak{M},$ and
let $\lambda $ be $\mathbb{T}-$measure on $\mathfrak{M}.$

a) There is a unique pair of $\mathbb{T}-$measures $\lambda ^{\prime
},\lambda ^{\prime \prime }$ on $\mathfrak{M}$ such that 
\begin{equation*}
\lambda =\lambda ^{\prime }+\lambda ^{\prime \prime }
\end{equation*}%
where $\lambda ^{\prime }\ll _{\mathbb{T}}\mu ,\lambda ^{\prime \prime
}\perp _{\mathbb{T}}\mu .$ If $\lambda $ is $\mathbb{D}-$finite measure on $%
\mathfrak{M}$ then $\lambda ^{\prime },\lambda ^{\prime \prime }$ are also
so.

b) For all $E\in \mathfrak{M}$ there is a unique $h\in L_{\mathbb{T}%
}^{1}\left( \mu \right) $ such that%
\begin{equation*}
\lambda ^{\prime }\left( E\right) =\dint\limits_{E}hd\mu .
\end{equation*}
\end{theorem}

\begin{proof}
Let $\mu =\mu _{1}\mathbf{e}_{1}+\mu _{2}\mathbf{e}_{2},\lambda =\lambda _{1}%
\mathbf{e}_{1}+\lambda _{2}\mathbf{e}_{2}\mathfrak{.}$

Since $\mu $ is a $\sigma -$finite $\mathbb{D}-$measure on $\mathfrak{M,}$ $%
\mu _{1}$\textbf{\ }and\textbf{\ }$\mu _{2}$ are positive $\sigma -$finite
measures on $\mathfrak{M.}$ Also $\lambda _{1}\mathbf{,}\lambda _{2}$ $\ $%
are complex measures on $\mathfrak{M.}$

a) Then for each $i=1,2$ there is a unique pair of complex measures $\lambda
_{i}^{\prime },\lambda _{i}^{\prime \prime }$ on $\mathfrak{M}$ such that 
\begin{equation*}
\lambda _{i}=\lambda _{i}^{\prime }+\lambda _{i}^{\prime \prime }
\end{equation*}%
where $\lambda _{i}^{\prime }$ is absolutely continuous with respect to $\mu
_{i}$ and $\lambda _{i}^{\prime \prime },\mu _{i}$ are mutually singular. If 
$\lambda $ is positive and finite measure on $\mathfrak{M}$ then $\lambda
_{i}^{\prime },\lambda _{i}^{\prime \prime }$ are also so.

Hence the result follows from these facts.

b) For each $i=1,2$ and for all $E\in \mathfrak{M}$ there is a unique $%
h_{i}\in L^{1}\left( \mu _{i}\right) $ such that 
\begin{equation*}
\lambda _{i}^{\prime }\left( E\right) =\dint\limits_{E}h_{i}d\mu _{i}.
\end{equation*}

Therefore for all $E\in \mathfrak{M}$ there is a unique $h=h_{1}\mathbf{e}%
_{1}+h_{2}\mathbf{e}_{2}\in L_{\mathbb{T}}^{1}\left( \mu \right) $ such that%
\begin{eqnarray*}
\lambda ^{\prime }\left( E\right) &=&\lambda _{1}^{\prime }\left( E\right) 
\mathbf{e}_{1}+\lambda _{2}^{\prime }\left( E\right) \mathbf{e}_{2} \\
&=&\left( \dint\limits_{E}h_{1}d\mu _{1}\right) \mathbf{e}_{1}+\left(
\dint\limits_{E}h_{2}d\mu _{2}\right) \mathbf{e}_{2} \\
&=&\dint\limits_{E}hd\mu .
\end{eqnarray*}
\end{proof}

\begin{theorem}
Let $\lambda $ be $\mathbb{T}-$measure on $\mathfrak{M}$ and $\mu $ be $%
\mathbb{D}-$measure on $\mathfrak{M}$. Then the following are equivalent:

a) $\lambda \ll _{\mathbb{T}}\mu .$

b) For every $\varepsilon \in \mathbb{D}^{+}/\left\{ 0\right\} $ there
exists $\delta \in \mathbb{D}^{+}/\left\{ 0\right\} $ such that $\left\vert
\lambda \left( E\right) \right\vert _{\mathbb{D}}\prec _{\mathbb{D}%
}\varepsilon $ for all $E\in \mathfrak{M}$ with $\left\vert \mu
(E)\right\vert _{\mathbb{D}}\prec _{\mathbb{D}}\delta .$
\end{theorem}

\begin{proof}
Let $\lambda =\lambda _{1}\mathbf{e}_{1}+\lambda _{2}\mathbf{e}_{2},\mu =\mu
_{1}\mathbf{e}_{1}+\mu _{2}\mathbf{e}_{2}$ where $\lambda _{1}\mathbf{,}%
\lambda _{2}$ are complex measures and $\mu _{1}\mathbf{,}\mu _{2}$ are
positive measures on $\mathfrak{M}$.

Then for $i=1,2$ the following two statements are equivalent

i) $\lambda _{i}$ is absolutely continuous with respect to $\mu _{i}$.

ii) For every $\varepsilon _{i}>0$ there exists $\delta _{i}>0$ such that $%
\left\vert \lambda _{i}\left( E\right) \right\vert <\varepsilon _{i}$ for
all $E\in \mathfrak{M}$ with $\mu _{i}\left( E\right) <\delta _{i}.$

If for $i=1,2,$ $\lambda _{i}$ is absolutely continuous with respect to $\mu
_{i}$ we get $\lambda \ll _{\mathbb{T}}\mu .$

Also if for every $\varepsilon _{i}>0$ there exists $\delta _{i}>0$ such
that $\left\vert \lambda _{i}\left( E\right) \right\vert <\varepsilon _{i}$
for all $E\in \mathfrak{M}$ with $\mu _{i}\left( E\right) <\delta _{i}$ we
can say for every $\varepsilon =\varepsilon _{1}\mathbf{e}_{1}+\varepsilon
_{2}\mathbf{e}_{2}\in \mathbb{D}^{+}/\left\{ 0\right\} $ there exists $%
\delta =\delta _{1}\mathbf{e}_{1}+\delta _{2}\mathbf{e}_{2}\in \mathbb{D}%
^{+}/\left\{ 0\right\} $ such that $\left\vert \lambda \left( E\right)
\right\vert _{\mathbb{D}}=\left\vert \lambda _{1}\left( E\right) \right\vert 
\mathbf{e}_{1}+\left\vert \lambda _{2}\left( E\right) \right\vert \mathbf{e}%
_{2}\prec _{\mathbb{D}}\varepsilon $ for all $E\in \mathfrak{M}$ with $%
\left\vert \mu (E)\right\vert _{\mathbb{D}}=\left\vert \mu _{1}\left(
E\right) \right\vert \mathbf{e}_{1}+\left\vert \mu _{2}\left( E\right)
\right\vert \mathbf{e}_{2}\prec _{\mathbb{D}}\delta .$
\end{proof}

\begin{theorem}
\label{T10}Let $\mathfrak{M}$ be a $\sigma -$algebra on $X.$ Let $\mu $ be $%
\mathbb{T}-$measure on $\mathfrak{M.}$ Then there exists a $\mathbb{T}-$%
measurable function $h$ such that $\left\vert h\left( x\right) \right\vert _{%
\mathbb{D}}=\mathbf{e}_{1}+\mathbf{e}_{2}$ for all $x\in X$ and such that%
\begin{equation*}
d\mu =hd\left\vert \mu \right\vert _{\mathbb{D}}.
\end{equation*}
\end{theorem}

\begin{proof}
Let $\mu =\mu _{1}\mathbf{e}_{1}+\mu _{2}\mathbf{e}_{2}.$

Since for $i=1,2,$ $\mu _{i}$ is a complex measure on the $\sigma -$algebra $%
\mathfrak{M}$ in $X,$ there exists measurable functions $h_{i}$ such that $%
\left\vert h_{i}\left( x\right) \right\vert =1$ for all $x\in X$ and such
that 
\begin{equation*}
d\mu _{i}=h_{i}d\left\vert \mu _{i}\right\vert .
\end{equation*}%
Setting $h=h_{1}\mathbf{e}_{1}+h_{2}\mathbf{e}_{2},$we get%
\begin{eqnarray*}
d\mu &=&d\mu _{1}\mathbf{e}_{1}+d\mu _{2}\mathbf{e}_{2} \\
&=&h_{1}d\left\vert \mu _{1}\right\vert \mathbf{e}_{1}+h_{2}d\left\vert \mu
_{2}\right\vert \mathbf{e}_{2} \\
&=&hd\left\vert \mu \right\vert _{\mathbb{D}}.
\end{eqnarray*}
\end{proof}

\begin{theorem}
Let $\mu $ be a $\mathbb{D}-$measure on $\mathfrak{M}$, $g\in L_{\mathbb{T}%
}^{1}\left( \mu \right) ,$ and for all $E\in \mathfrak{M}$ 
\begin{equation*}
\lambda \left( E\right) =\dint\limits_{E}gd\mu .
\end{equation*}%
Then for all $E\in \mathfrak{M,}$ 
\begin{equation*}
\left\vert \lambda \right\vert _{\mathbb{D}}\left( E\right)
=\dint\limits_{E}\left\vert g\right\vert _{\mathbb{D}}d\mu .
\end{equation*}

\begin{proof}
Let $\mu =\mu _{1}\mathbf{e}_{1}+\mu _{2}\mathbf{e}_{2},$ $\lambda =\lambda
_{1}\mathbf{e}_{1}+\lambda _{2}\mathbf{e}_{2}$ and $g=g_{1}\mathbf{e}%
_{1}+g_{2}\mathbf{e}_{2}.$ Then $\mu _{1},\mu _{2}$ are positive measures on 
$\mathfrak{M.}$

Since $g\in L_{\mathbb{T}}^{1}\left( \mu \right) ,$ then $g_{1}\in
L^{1}\left( \mu _{1}\right) $ and $g_{2}\in L^{1}\left( \mu _{2}\right) .$

Also $\lambda \left( E\right) =\dint\limits_{E}gd\mu \Rightarrow \lambda
_{1}\left( E\right) =\dint\limits_{E}g_{1}d\mu _{1}$ and $\lambda _{2}\left(
E\right) =\dint\limits_{E}g_{2}d\mu _{2}.$

Then by Consequences of the Radon-Nikodym Theorem \cite{Ru}, we have%
\begin{equation*}
\left\vert \lambda _{1}\right\vert \left( E\right)
=\dint\limits_{E}\left\vert g_{1}\right\vert d\mu _{1}\text{ and }\left\vert
\lambda _{2}\right\vert \left( E\right) =\dint\limits_{E}\left\vert
g_{2}\right\vert d\mu _{2}.\text{ }
\end{equation*}

Therefore 
\begin{eqnarray*}
\left\vert \lambda \right\vert _{\mathbb{D}}\left( E\right) &=&\left\vert
\lambda _{1}\right\vert \left( E\right) \mathbf{e}_{1}+\left\vert \lambda
_{1}\right\vert \left( E\right) \mathbf{e}_{2} \\
&=&\left( \dint\limits_{E}\left\vert g_{1}\right\vert d\mu _{1}\right) 
\mathbf{e}_{1}+\left( \dint\limits_{E}\left\vert g_{2}\right\vert d\mu
_{2}\right) \mathbf{e}_{2} \\
&=&\dint\limits_{E}(\left\vert g_{1}\right\vert \mathbf{e}_{1}+\left\vert
g_{1}\right\vert \mathbf{e}_{2})d(\mu _{1}\mathbf{e}_{1}+\mu _{2}\mathbf{e}%
_{2}) \\
&=&\dint\limits_{E}\left\vert g\right\vert _{\mathbb{D}}d\mu .
\end{eqnarray*}
\end{proof}
\end{theorem}

\begin{theorem}[Hahn Decomposition Theorem]
Let $\mathfrak{M}$ be a $\sigma -$algebra on $X.$ Let $\mu $ be a $\mathbb{D}%
^{+}-$measure on $\mathfrak{M}$. Then there exists a partition $%
\{A,B,C,D\}\subset \mathfrak{M}$ of $X\ $such that for all $E\in \mathfrak{M,%
}$%
\begin{equation*}
\mu ^{+}\left( E\right) =\mu \left( E\cap A\right) +\mathbf{e}_{1}\left\vert
\mu \left( E\cap C\right) \right\vert _{\mathbb{D}}+\mathbf{e}_{2}\left\vert
\mu \left( E\cap D\right) \right\vert _{\mathbb{D}},
\end{equation*}%
\begin{equation}
\mu ^{-}\left( E\right) =-\mu \left( E\cap B\right) -\mu \left( E\cap
C\right) -\mu \left( E\cap D\right) +\mathbf{e}_{1}\left\vert \mu \left(
E\cap C\right) \right\vert _{\mathbb{D}}+\mathbf{e}_{2}\left\vert \mu \left(
E\cap D\right) \right\vert _{\mathbb{D}}.  \label{a}
\end{equation}

\begin{proof}
By Theorem $\ref{T10},$ $d\mu =hd\left\vert \mu \right\vert _{\mathbb{D}},$
where $\left\vert h\left( x\right) \right\vert _{\mathbb{D}}=\mathbf{e}_{1}+%
\mathbf{e}_{2}$ for all $x\in X.$ Since $\mu $ is hyperbolic, it follows
that $h$ is hyperbolic, hence $h(x)=\pm \mathbf{e}_{1}\pm \mathbf{e}_{2}$
for all $x\in X.$ Put 
\begin{equation*}
A=\{x:h(x)=\mathbf{e}_{1}+\mathbf{e}_{2}\},
\end{equation*}%
\begin{equation*}
B=\{x:h(x)=-\mathbf{e}_{1}-\mathbf{e}_{2}\},
\end{equation*}%
\begin{equation*}
C=\{x:h(x)=\mathbf{e}_{1}-\mathbf{e}_{2}\},
\end{equation*}%
\begin{equation*}
D=\{x:h(x)=-\mathbf{e}_{1}+\mathbf{e}_{2}\}.
\end{equation*}

Since $\mu ^{+}=\frac{1}{2}\left( \left\vert \mu \right\vert _{\mathbb{D}%
}+\mu \right) ,$ and since%
\begin{equation*}
\frac{1}{2}(1+h)=\left\{ 
\begin{array}{c}
h\text{ \ on }A, \\ 
0\text{ \ on }B, \\ 
\mathbf{e}_{1}\text{ \ on }C, \\ 
\mathbf{e}_{2}\text{ on }D,%
\end{array}%
\right.
\end{equation*}

we have, for any $E\in \mathfrak{M,}$%
\begin{eqnarray*}
\mu ^{+}\left( E\right) &=&\frac{1}{2}\int\limits_{E}(1+h)d\left\vert \mu
\right\vert _{\mathbb{D}} \\
&=&\int\limits_{E\cap A}hd\left\vert \mu \right\vert _{\mathbb{D}%
}+\int\limits_{E\cap C}\mathbf{e}_{1}d\left\vert \mu \right\vert _{\mathbb{D}%
}+\int\limits_{E\cap D}\mathbf{e}_{2}d\left\vert \mu \right\vert _{\mathbb{D}%
} \\
&=&\mu (E\cap A)+\mathbf{e}_{1}\left\vert \mu (E\cap C)\right\vert _{\mathbb{%
D}}+\mathbf{e}_{2}\left\vert \mu (E\cap D)\right\vert _{\mathbb{D}}.
\end{eqnarray*}

Since $\mu (E)=$ $\mu (E\cap A)+$ $\mu (E\cap B)+$ $\mu (E\cap C)+$ $\mu
(E\cap D)$ and since $\mu =\mu ^{+}-\mu ^{-},$ (\ref{a}) follows.
\end{proof}
\end{theorem}

\subsection{Hyperbolic Invariant Measure}

\begin{definition}
Let $\mathfrak{B}$ be a Borel $\sigma -$algebra on a metric space $X.$ A $%
\mathbb{D}-$measure $\mu $ on $X$ is $\mathbb{D}-$finite if $\mu \left(
X\right) \prec _{\mathbb{D}}\infty _{\mathbb{D}}$ and $\mu $ is called a
Borel $\mathbb{D}-$probability measure if $\mu \left( X\right) =\mathbf{e}%
_{1}$ $+$ $\mathbf{e}_{2}$ or $\mathbf{e}_{1}$ or $\mathbf{e}_{2}.$
\end{definition}

For any $\mathbb{D}-$finite, nonzero, $\mathbb{D}-$measure $\widehat{\mu }$
on $X$ we may define a Borel $\mathbb{D}-$probability measure as%
\begin{equation*}
\mu \left( A\right) =\frac{\widehat{\mu }\left( A\right) }{\widehat{\mu }%
\left( X\right) }
\end{equation*}%
for all $A\in \mathfrak{B}.$

Throughout this section we will need a set $X$ equipped with a Borel $\sigma
-$algebra $\mathfrak{B}$ and a measurable function $f:X\rightarrow X.$ We
denote by $\mathcal{M}_{\mathbb{D}}$, the space of all Borel $\mathbb{D}-$%
probability measures on $X$.

\begin{definition}
A Borel $\mathbb{D}-$probability measure $\mu =\mu _{1}\mathbf{e}_{1}+\mu
_{2}\mathbf{e}_{2}$ on $X$ is said to be $\mathbb{D}-$invariant with respect
to a measurable function $f:X\rightarrow X$ if $\mu _{1}$ and $\mu _{2}$ are
invariant i.e., $\mu _{1}\left( f^{-1}\left( A\right) \right) =\mu
_{1}\left( A\right) $ and $\mu _{2}\left( f^{-1}\left( A\right) \right) =\mu
_{2}\left( A\right) $ for all $A\in \mathfrak{B}.$
\end{definition}

\begin{definition}[Push-forward of measures]
Let $f_{\ast }:\mathcal{M}_{\mathbb{D}}\longrightarrow \mathcal{M}_{\mathbb{D%
}}$ be the map from the space of $\mathbb{D}-$probability measures to
itself, defined by%
\begin{equation*}
f_{\ast }\mu \left( A\right) :=\mu \left( f^{-1}\left( A\right) \right) =\mu
_{1}\left( f^{-1}\left( A\right) \right) \mathbf{e}_{1}+\mu _{2}\left(
f^{-1}\left( A\right) \right) \mathbf{e}_{2}=f_{\ast }\mu _{1}\left(
A\right) \mathbf{e}_{1}+f_{\ast }\mu _{2}\left( A\right) \mathbf{e}_{2}.
\end{equation*}
\end{definition}

We call $f_{\ast }\mu $ the push forward of $\mu $ by $f.$

The mapping is well defined and $\mu $ is invarient iff $f_{\ast }\mu =\mu .$
For any $\mu \in \mathcal{M}_{\mathbb{D}}$ and any $i\geq 1$ we also let 
\begin{equation*}
f_{\ast }^{i}\mu \left( A\right) :=\mu \left( f^{-i}\left( A\right) \right) .
\end{equation*}

We now prove some properties of the map $f_{\ast }.$

\begin{lemma}
\label{L1}For all $\varphi \in L_{\mathbb{T}}^{1}\left( \mu \right) $ we
have $\int \varphi d(f_{\ast }\mu )=\int \varphi \circ fd\mu .$

\begin{proof}
First let $\varphi =\chi _{A}$ be the characteristic function of some $%
A\subseteq X.$ Then 
\begin{eqnarray*}
\int \chi _{A}d(f_{\ast }\mu ) &=&f_{\ast }\mu (A) \\
&=&\mu \left( f^{-1}\left( A\right) \right) \\
&=&\mu _{1}\left( f^{-1}\left( A\right) \right) \mathbf{e}_{1}+\mu
_{2}\left( f^{-1}\left( A\right) \right) \mathbf{e}_{2} \\
&=&\left( \int \chi _{f^{-1}(A)}d\mu _{1}\right) \mathbf{e}_{1}+\left( \int
\chi _{f^{-1}(A)}d\mu _{2}\right) \mathbf{e}_{2} \\
&=&\int (\chi _{A}\circ f)d\mu .
\end{eqnarray*}

So, the statement is true for characteristic functions and thus for simple
functions. Hence the statement is true for general integrable functions by
standard approximation arguments.
\end{proof}
\end{lemma}

\begin{corollary}
$f_{\ast }:\mathcal{M}_{\mathbb{D}}\longrightarrow \mathcal{M}_{\mathbb{D}}$
is continuous.

\begin{proof}
Consider a sequence $\mu _{n}\longrightarrow \mu $ in $\mathcal{M}_{\mathbb{D%
}}$. Then, by Lemma $\ref{L1}$, for any continuous function $\varphi
:X\longrightarrow \mathbb{D}$ we have%
\begin{equation*}
\int \varphi d(f_{\ast }\mu _{n})=\int \varphi \circ fd\mu
_{n}\longrightarrow \int \varphi \circ fd\mu =\int \varphi d(f_{\ast }\mu )
\end{equation*}

which implies $f_{\ast }\mu _{n}\longrightarrow f_{\ast }\mu .$ Hence $%
f_{\ast }$ is continuous.
\end{proof}
\end{corollary}

\begin{corollary}
\label{C2}$\mu $ is $\mathbb{D-}$invariant with respect to a measurable
function $f:X\longrightarrow X$ if and only if $\int \varphi \circ fd\mu
=\int \varphi d\mu $ for all continuous function $\varphi :X\longrightarrow 
\mathbb{D}$.

\begin{proof}
$\varphi =\varphi _{1}\mathbf{e}_{1}+\varphi _{2}\mathbf{e}_{2}$ is
continuous $\Longleftrightarrow \varphi _{1}$ and $\varphi _{2}$ are
continuous.

Suppose first that $\mu $ is $\mathbb{D-}$invariant. Then $f_{\ast }\mu =\mu
.$ Now using Lemma $\ref{L1},$ we have%
\begin{equation*}
\int \varphi \circ fd\mu =\int \varphi df_{\ast }\mu =\int \varphi d\mu .
\end{equation*}

For the converse, we have 
\begin{equation*}
\int \varphi d\mu =\int \varphi \circ fd\mu =\int \varphi d(f_{\ast }\mu ).
\end{equation*}

which implies 
\begin{equation*}
\left( \int \varphi _{1}d\mu _{1}\right) \mathbf{e}_{1}+\left( \int \varphi
_{2}d\mu _{2}\right) \mathbf{e}_{2}=\left( \int \varphi _{1}df_{\ast }\mu
_{1}\right) \mathbf{e}_{1}+\left( \int \varphi _{2}df_{\ast }\mu _{2}\right) 
\mathbf{e}_{2}.
\end{equation*}

So,%
\begin{equation*}
\int \varphi _{1}d\mu _{1}=\int \varphi _{1}df_{\ast }\mu _{1}
\end{equation*}

and%
\begin{equation*}
\int \varphi _{2}d\mu _{2}=\int \varphi _{2}df_{\ast }\mu _{2}
\end{equation*}

for every continuous function $\varphi _{1},\varphi _{2}:X\longrightarrow 
%TCIMACRO{\U{211d} }%
%BeginExpansion
\mathbb{R}
%EndExpansion
.$

By the Riesz Representation Theorem, measures correspond to linear
functionals and therefore this can be restated as saying that 
\begin{equation*}
\mu _{1}(\varphi _{1})=f_{\ast }\mu _{1}(\varphi _{1})
\end{equation*}

and%
\begin{equation*}
\mu _{2}(\varphi _{2})=f_{\ast }\mu _{2}(\varphi _{2})
\end{equation*}

for all continuous function $\varphi _{1},\varphi _{2}:X\longrightarrow 
%TCIMACRO{\U{211d} }%
%BeginExpansion
\mathbb{R}
%EndExpansion
.$

Hence 
\begin{equation*}
\mu _{1}=f_{\ast }\mu _{1}\text{ and }\mu _{2}=f_{\ast }\mu _{2}
\end{equation*}

and so 
\begin{equation*}
f_{\ast }\mu =\mu .
\end{equation*}

Therefore $\mu $ is $\mathbb{D-}$invariant.
\end{proof}
\end{corollary}

We now prove a general result which gives conditions to guarantee that
atleast some $\mathbb{D-}$invariant $\mathbb{D}-$probability measure exists.

Let 
\begin{equation*}
\mathcal{M}_{\mathbb{D}}\left( f\right) =\{\mu \in \mathcal{M}_{\mathbb{D}}%
\text{: }\mu \text{ is }\mathbb{D-}\text{invariant with respect to }f\text{%
\}.}
\end{equation*}

\begin{theorem}
Suppose $M$ is compact metric space and $f:M\longrightarrow M$ is
continuous. Then $\mathcal{M}_{\mathbb{D}}\left( f\right) $ is non-empty,
convex, compact.

\begin{proof}
Krylov-Boguliobov Theorem $\cite{Luz}$ states that if $M$ is compact metric
space and $f:M\longrightarrow M$ is continuous, then the set $\{\mu ^{\ast
}\in \mathcal{M}{}^{\ast }$: $\mu ^{\ast }$ is $f-$invariant\} is non-empty,
where $\mu ^{\ast }$ is probability measure and $\mathcal{M}{}^{\ast }$ is
the space of probability measures.

It is clear that 
\begin{equation*}
\{\mu ^{\ast }\in \mathcal{M}{}^{\ast }:\mu ^{\ast }\text{ is }f-\text{
invariant}\}\subseteq \mathcal{M}_{\mathbb{D}}\left( f\right) .
\end{equation*}

Hence $\mathcal{M}_{\mathbb{D}}\left( f\right) $ is non-empty.

Now let $\mu ^{1},\mu ^{2}\in \mathcal{M}_{\mathbb{D}}\left( f\right) $ .
Then $\mu ^{1}(f^{-1}(A))=\mu ^{1}(A)$ and $\mu ^{2}(f^{-1}(A))=\mu ^{2}(A),$
for all $A\in \mathfrak{B}.$

Now for $t\in \lbrack 0,1],$ let $\mu =t\mu ^{1}+(1-t)\mu ^{2}.$

Then 
\begin{eqnarray*}
\mu (f^{-1}(A)) &=&t\mu ^{1}(f^{-1}(A))+(1-t)\mu ^{2}(f^{-1}(A)) \\
&=&t\mu ^{1}(A)+(1-t)\mu ^{2}(A) \\
&=&\mu (A).
\end{eqnarray*}

Hence $\mathcal{M}_{\mathbb{D}}\left( f\right) $ is convex.

To show compactness, suppose that $\mu _{n}$ is a sequence in $\mathcal{M}_{%
\mathbb{D}}\left( f\right) $ \ converging to some $\mu \in $ $\mathcal{M}_{%
\mathbb{D}}$. Then by Lemma $\ref{L1}$ we have, for any continuous function $%
\varphi ,$ that $\int \varphi \circ fd\mu =\lim\limits_{n\rightarrow \infty
}\int \varphi \circ fd\mu _{n}=\lim\limits_{n\rightarrow \infty }\int
\varphi d\mu _{n}=\int \varphi d\mu .$ Therefore by Corollary $\ref{C2}$, $%
\mu $ is $\mathbb{D-}$invariant and so $\mu \in $ $\mathcal{M}_{\mathbb{D}%
}(f)$.
\end{proof}
\end{theorem}

\end{document}